\documentclass{amsart}
\usepackage{amsmath, amssymb, mathrsfs, verbatim, multirow}
\usepackage[all]{xy}
\usepackage{pifont}
\usepackage{float}
\usepackage{color}
\usepackage{enumitem}

\usepackage{tikz}
\usepackage{tikz-cd}
\usepackage{tkz-tab}
\usepackage{subcaption} 

\newtheorem{Teo}{Theorem}[section]
\newtheorem{Prop}[Teo]{Proposition}
\newtheorem{Lema}[Teo]{Lemma}
\newtheorem{Cor}[Teo]{Corollary}

\theoremstyle{definition}
\newtheorem{Def}[Teo]{Definition}

\newtheorem{Obs}[Teo]{Remark}

\newcommand{\Q}{\mathbb{Q}}

\newcommand{\Z}{\mathbb{Z}}
\newcommand{\N}{\mathbb{N}}

\newcommand{\lra}{\longrightarrow}
\newcommand{\Lra}{\Longrightarrow}
\newcommand{\VR}{\mathcal{O}}

\newcommand{\MI}{\mathfrak{m}}

\newcommand{\SU}{\mbox{\rm supp}}

\newcommand{\K}{\mathbb{K}}

\DeclareMathOperator{\inv}{in}
\begin{document}
\title{On truncations of valuations}
\author{Novacoski, J. A. and Silva de Souza, C. H.}
\thanks{During the realization of this project the authors were supported by grants from Funda\c c\~ao de Amparo \`a Pesquisa do Estado de S\~ao Paulo (process numbers 2017/17835-9 and 2020/05148-0).}

\begin{abstract}
In this paper we study the truncation $\nu_q$ of a valuation $\nu$ on a polynomial $q$. It is known that when $q$ is a key polynomial, then $\nu_q$ is a valuation. It is also known that the converse does not hold. We show that when $q$ is a key polynomial, then $\nu_q$ is the restriction of the truncation given by an optimizing root of $q$. We also discuss which conditions assure that $\nu_q=\nu$. Finally, we assume that $\nu_q$ is a valuation and present some conditions, given in terms of the graded algebra, to assure that $q$ is a key polynomial.
\end{abstract}


\maketitle
\section{Introduction}
Fix a valuation  $\nu$ in $\K[x]$, the ring of polynomials on one indeterminate over the field $\K$. Let $q\in \K[x]$ be a non-constant polynomial. Then there exist, uniquely determined, polynomials $f_0,\ldots, f_s\in \K[x]$ with $\deg(f_i)<\deg(q)$ for every $i$, $0\leq i\leq s$, such that
$$f = f_0+f_1q+\ldots +f_sq^s.$$
We call this expression the $q$-\textbf{expansion} of $f$. Hence, we can construct a map given by
$$\nu_q(f):=\underset{0\leq i\leq s}{\min} \{ \nu(f_iq^i)  \},$$
and call it the \textbf{truncation} of $\nu$ at $q$. This map is not always a valuation, as we can see in Example 2.4 of \cite{josneiKeyPolyPropriedades}. If $Q$ is a key polynomial, then $\nu_Q$ is a valuation in $\K[x]$ (see Proposition 2.6 of \cite{josneiKeyPolyMinimalPairs}). 

In this paper we extend  some results of  \cite{josneiKeyPolyPropriedades}, \cite{josneiKeyPolyMinimalPairs}  and \cite{josneimonomial} about valuations given by truncations, key polynomials, minimal pairs, optimizing roots and valuation-transcendental valuations. 

The first result, that we prove in Section~\ref{SecResultado1}, deals with an equality of valuations given by truncations on different rings. Fix an algebraic closure $\overline{\K}$ of $\K$, a Krull valuation $\nu$ on $\K[x]$ and an extension $\mu$ of $\nu$ to $\overline{\K}[x]$. Take a key polynomial $Q$ for $\nu$ and an optimizing root $a$ of $Q$ (i.e., a root $a$ of $Q$ for which $\mu(x-a)$ is maximum). Our first result (Theorem \ref{teorandretromeno}) says that for every $f\in\K[x]$ we have
\[
\nu_Q(f)=\mu_{x-a}(f).
\]

This result is not new and was proved, in a similar context, in \cite{popescu2} and \cite{popescuSurLaDefinition}. Our proof follows their step and is an adaptation to our definition of key polynomials. In the process, we present some intermediate results that are important on their own, as well as alternative simpler proofs. Also, recently Bengus-Lasnier presented a proof of this result in \cite{Andrei}. He uses the structure of the graded algebra of a valuation, while our proof only deals with basic properties of valuations. In Section \ref{gradegalh} we present a brief comparison of our proof and the one by Bengus-Lasnier.

The second main result is presented in Section~\ref{SecResultado2}. It is a generalization of Theorem 1.3 of \cite{josneiKeyPolyMinimalPairs}. Let $\nu$ be a valuation on $\K[x]$. Then $\nu$ is called \textbf{value-transcendental} if it is not Krull or if the quotient group $\nu\K(x)/\nu\K$ is not a torsion group. On the other hand, we say that $\nu$ is \textbf{residue-transcendental} if the field extension $\K(x)\nu\mid \K\nu$ is transcendental. By  the Abhyankar inequality, we see that a valuation cannot be  value-transcendental and residue-transcendental at the same time. A valuation that is of one of the previous types is called \textbf{valuation-transcendental}. 

Our second main result is the following.
\begin{Teo}\label{teoValTransiffTruncamento}
A valuation $\nu$ is valuation-transcendental if and only if there exists a key polynomial $Q$ such that $\nu=\nu_Q$.
\end{Teo}

In \cite{josneiKeyPolyMinimalPairs}, it is shown that $\nu$ is valuation-transcendental if and only if there exists a polynomial $q$ such that $\nu=\nu_q$. Hence, we have to show that if $\nu=\nu_q$ for some polynomial $q$, then there exists a key polynomial $Q$ such that $\nu=\nu_Q$. We point out that, if $\nu$ is residue-transcendental, it is shown in \cite{Andrei} that $\nu=\nu_Q$ for some key polynomial $Q$.

It is important to have that $\nu$ is the truncation on a key polynomial for many reasons. For instance, it is known that the graded algebra $\mathcal G_Q$ of a truncation on a key polynomial $Q$ has a simple structure (see the discussion on Section \ref{gradegalh}). This simple structure allows to study, for instance, irreducibility of elements in  $\mathcal G_Q$.

This paper is  divided as follows. In Section \ref{Preli}, we present the main definitions and results that will be used to prove our results. In Section \ref{SecResultado1}, we prove the result about the restriction of $\mu_{x-a}$ to $\K[x]$. The main goal of Section \ref{SecResultado2} is to show Theorem \ref{teoValTransiffTruncamento}. Finally, in Section \ref{gradegalh}, we present some results on the graded algebra associated to a truncation $\nu_q$, when it is a valuation. In particular, we give a brief description of the proof of Bengus-Lasnier of the result on the restriction of the truncation (in terms of graded algebras).

\section{Preliminaries}\label{Preli}
\begin{Def}
Take a commutative ring $R$ with unity. A \index{Valuation}\textbf{valuation} on $R$ is a mapping $\nu:R\lra \Gamma_\infty :=\Gamma \cup\{\infty\}$ where $\Gamma$ is an ordered abelian group (and the extension of addition and order to $\infty$ in the obvious way), with the following properties:
\begin{description}
	\item[(V1)] $\nu(ab)=\nu(a)+\nu(b)$ for all $a,b\in R$.
	\item[(V2)] $\nu(a+b)\geq \min\{\nu(a),\nu(b)\}$ for all $a,b\in R$.
	\item[(V3)] $\nu(1)=0$ and $\nu(0)=\infty$.
\end{description}
\end{Def}

Let $\nu: R \lra\Gamma_\infty$ be a valuation. The set $\SU(\nu)=\{a\in R\mid \nu(a )=\infty\}$ is called the \textbf{support of $\nu$}.  A valuation $\nu$ is a \index{Valuation!Krull}\textbf{Krull valuation} if $\SU(\nu)=\{0\}$. The \textbf{value group of $\nu$} is the subgroup of $\Gamma$ generated by 
$\{\nu(a)\mid a \in R\setminus \SU(\nu) \}
$ and is denoted by $\nu R$. If $R$ is a field, then we define the \textbf{valuation ring} of $\nu$  by $\VR_\nu:=\{ a\in R\mid \nu(a)\geq 0 \}$. The ring $\VR_\nu$ is a local ring with unique maximal ideal $\MI_\nu:=\{a\in R\mid \nu(a)>0 \}. $ We define the \textbf{residue field} of $\nu$ to be the field $\VR_\nu/\MI_\nu$ and denote it by $R\nu$. The image of $a\in \VR_\nu$ in $R\nu$ is denoted by $a\nu$. 

Let $\nu$ be a valuation on $\K[x]$,  the ring of polynomials in one indeterminate over the field $\K$. Let $q\in \K[x]$ be a non-constant polynomial.

\begin{Lema}\label{lemMuQLivreTorcaoMuigualMQ}
Suppose that  $q\in\K[x]$ is a polynomial such that $\nu(q)$ is torsion-free over  $\nu\K$ and for every  $f\in\K[x]$ with $\deg(f)<\deg(q)$ we have that $\nu(f)$ is torsion over  $\nu\K$. Then $\nu=\nu_q$.
\end{Lema}

\begin{proof}
	For any $f\in \K[x]$, let $f = f_0 + \ldots + f_s q^s$ be its $q$-expansion.  For $i,j$, $0\leq i<j\leq s$, we claim that $\nu(f_iq^i)\neq \nu(f_jq^j)$. Otherwise, we would have
	$$\nu(f_i)+i\nu(q) = \nu(f_j)+j\nu(q)$$
	which implies 
	$$ (i-j)\nu(q) = \nu(f_j)-\nu(f_i).$$
	However, since $\deg(f_i),\deg(f_j)<\deg(q)$, we have by hypothesis that $\nu(f_i)$ and $\nu(f_j)$ are torsion over $\nu\K$. This would imply that $\nu(q)$ is torsion over $\nu\K$, which is a contradiction to our assumption on $\nu(q)$. Thus $\nu(f_iq^i)\neq \nu(f_jq^j)$ if $i\neq j$. Hence,
	$$\nu(f) = \nu(f_0 + \ldots + f_s q^s)=\underset{0\leq i \leq s}{\min}\{ \nu(f_iq^i)  \}=\nu_q(f).$$
	
\end{proof}

Our main definition of key polynomial relates to the one in \cite{josneimonomial}. We first define 
$$\alpha = \begin{cases}
\min\{\deg(g) \mid g\in \SU(\nu)\} \text{ if } \SU(\nu)\neq \{ 0\},\\
\qquad \qquad\infty \hspace{2.3cm} \text{ if } \SU(\nu)= \{ 0\}.
\end{cases} $$

Let $f\in \K[x]$ be a non-zero polynomial. For every $b\in \N$, we consider the formal \textbf{Hasse-derivative of order} $b$ of $f(x)=a_0+\ldots+a_sx^s$, defined by
$$\partial_bf(x):= \sum_{i=b}^s\binom{i}{b}a_ix^{i-b}\in \K[x].  $$

\begin{Def}
	Let $f\in \K[x]$ be a non-zero polynomial. 
	\begin{itemize}
		\item If  $\deg(f)<\alpha$ and $\deg(f)\neq 0$, then
		$$\epsilon(f):=\underset{1\leq b\leq \deg(f)}{\max}\left\lbrace\left.\frac{\nu(f) - \nu(\partial_bf)}{b}  \;\right| \nu(f),\nu(\partial_bf)\in\Gamma \right\rbrace\in \Gamma \otimes_\Z \Q. $$
		
		\item If $\deg(f)=0$, then $\epsilon(f):=-\infty$.
		
		\item If $f$ is a generator of $\SU(\nu)$, then $\epsilon(f):= \infty$.
	\end{itemize}

\end{Def}

\begin{Def}\label{defiPoliChave}
	A monic polynomial $Q\in \K[x]$ is a  \textbf{key polynomial} of level $\epsilon(Q)$ if, for every $f\in \K[x]$, then
	$$\epsilon(f)\geq \epsilon(Q) \Longrightarrow \deg(f)\geq \deg(Q). $$
\end{Def}

Let $\nu$ be a valuation on $\K[x]$ and take  $\mu$ an extension of $\nu$ to $\overline{\K}[x]$.

\begin{Def}
	A \textbf{minimal pair} for $\mu$ is a pair $(a, \delta)\in \overline{\K}\times \mu\overline{\K}[x]$ such that, for all $b\in \overline{\K}$, we have
	$$ \mu(b-a)\geq \delta \Lra [\K(b):\K]\geq [\K(a):\K].$$
\end{Def}

\begin{Def}
	
	Let $f\in \K[x]$ be a non-constant polynomial and suppose that $\mu(x-a)\neq \infty$ for some root $a\in \overline{\K}$ of $f$. We define
	$$\delta(f):=\max\{\mu(x-a)\mid a\in \overline{\K} \text{ and }f(a)=0\}. $$
	
	A root $a$ of $f$ such that $\delta(f)=\mu(x-a)$ is called an \textbf{optimizing root} of $f$. 
\end{Def}

\begin{Def}
Let $(\K,v)$ be any valued field and take $\delta$ in some extension of the value group $v\K$. Then the map
\[
\nu(a_0+a_1x+\ldots+a_sx^s)=\min_{0\leq i\leq s}\{v(a_i)+i\delta\}
\]
is a valuation on $\K[x]$  (see for instance Corollary 2.4 of \cite{josneimonomial}). This valuation is called the \textbf{monomial valuation}, with respect to $x$, obtained by $\delta$ and $v$.
\end{Def}

If $\mu$ is a valuation on $\overline{\K}[x]$, $a\in \overline{\K}$ and $\delta=\mu(x-a)$ we will denote by $\mu_{a,\delta}$ the monomial valuation, with respect to $x-a$, obtained by $\delta$ and $\mu|_{\overline\K}$, i.e., 
$$\mu_{a,\delta}\left( \sum_{i=0}^r a_i (x-a)^i \right):= \underset{0\leq i \leq r}{\min} \{\mu(a_i)+i\delta\}. $$

\begin{Def}
A pair $(a, \delta)\in \overline{\K}\times \mu(\K[x])$ such that $\mu= \mu_{a,\delta}$ is called a \textbf{pair of definition} for $\mu$. We say that a pair $(a,\delta)$ is a \textbf{minimal pair of definition}  if $(a,\delta)$ is a minimal pair and a pair of definition for $\mu$.	
\end{Def}

Suppose that $\mu$ is a valuation on $\overline{\K}[x]$ extending a valuation $\nu$ on $\K[x]$.
\begin{Lema}\label{Lemaintermediario}
If $\mu(a-c)\geq\mu(x-a)$, then $\mu_{x-a}(x-c)=\mu(x-c)$.
\end{Lema}
\begin{proof}
Since $x-c=(x-a)+(a-c)$ we have
\begin{equation}\label{eqobreopaimin}
\mu_{x-a}(x-c)=\min\{\mu(a-c),\mu(x-a)\}.
\end{equation}
	
If $\mu(x-c)<\delta=\mu(x-a)$, then
\[
\mu(a-c)=\min\{\mu(x-a),\mu(x-c)\}=\mu(x-c)<\mu(x-a).
\]
This and \eqref{eqobreopaimin} give us
\[
\mu_{x-a}(x-c)=\mu(a-c)=\mu(x-c)
\]
	
If $\mu(x-c)=\mu(x-a)$, then
\[
\mu(a-c)\geq\min\{\mu(x-a),\mu(x-c)\}=\mu(x-a).
\]
Hence, by \eqref{eqobreopaimin}, we obtain
\[
\mu_{x-a}(x-c)=\mu(x-c).
\]
\end{proof}
\begin{Lema}\label{lemParDeDefinicaoCondicao}
A pair $(a,\delta)$ is a pair of definition for $\mu$ if and only if $\delta = \mu(x-a)\geq \mu(x-c)$ for every  $c\in \overline{\K}$. 
\end{Lema}

\begin{proof}
Assume that $(a,\delta)$ is a pair of definition. Then, for every $c\in \overline{\K}$, we have
	$$\mu(x-c)=\mu_{a,\delta}(x-c)=\min\{ \mu(x-a),\mu(a-c)\}\leq \mu(x-a)=\delta. $$

For the converse, suppose that $\delta=\mu(x-a)\geq \mu(x-c)$ for every $c\in \overline{\K}$. Since every polynomial in $\overline\K[x]$ can be written as a product of degree one factors, it is enough to show that
\[
\mu(x-c)=\mu_{a,\delta}(x-c)
\]
for every $c\in \overline \K$. This is an immediate consequence of Lemma \ref{Lemaintermediario} and our assumption. 
\end{proof}

If $(a, \delta)$ is a pair of definition for $\mu$, then there might exist other pairs of definiton for $\mu$. The next result presents a way to relate these pairs. 
\begin{Lema}\label{lemParesDefinemMesmaVal}
	Two pairs $(a,\delta)$ e $(a', \delta')$ define the same monomial valuation if and only if  $\delta = \delta'$ and $\mu(a-a')\geq \delta$. 
\end{Lema}

\begin{proof}
	
	First we suppose that  $(a,\delta)$ and $(a', \delta')$ define the same monomial valuation. Then
	$$\delta' = \mu_{a',\delta'}(x-a') = \mu_{a,\delta}(x-a') = \min\{\delta, \mu(a-a')\}.$$
	
	By a symmetric argument we see that $\delta = \min\{ \delta', \mu(a-a')  \}$. Thus $\delta=\delta'$ and $\mu(a-a')\geq \delta$.
	
	Now we will show that if $\delta'=\delta$ and $\mu(a-a')\geq \delta$, then $\mu_{a',\delta'}=\mu_{a,\delta}$. It is enough to prove that they conincide on monic linear polynomials. For $b\in \overline{\K}$, we have $\mu_{a',\delta'}(b)=\mu(b) = \mu_{a,\delta}$. Now, for $x-b\in \overline{\K}[x]$ we have
	$$ \mu_{a',\delta}(x-b) =\min\{ \delta, \mu(a-b)  \} \text{ and } \mu_{a,\delta}(x-b) =\min\{ \delta, \mu(a'-b)  \}. $$
	If $\mu(a-b)\geq \delta$, then $ \mu_{a,\delta}(x-b) = \delta$ and
	$$ \mu(a'-b) = \mu(a'-a+a-b)\geq \min\{ \mu(a'-a), \mu(a-b)\}\geq \delta.$$
	Hence $$\mu_{a,\delta}(x-b) = \delta = \mu_{a',\delta}(x-b).$$
	
	If $\mu(a-b)<\delta$, then
	$$\mu(a'-b) = \mu(a'-a+a-b) =  \mu(a-b).$$
	Hence, $$\mu_{a,\delta}(x-b) = \mu(a-b)=\mu(a'-b) = \mu_{a',\delta}(x-b).$$
	
\end{proof}
Let $S$ be the set
$$S := \{ a\in\overline{\K}\mid \mu=  \mu_{a,\delta} \}.$$
\begin{Obs}\label{obsbSRaizOtimiza}
	If $a\in S$, then by Lemma~\ref{lemParDeDefinicaoCondicao} we see that $a$ is an optimizing root of its minimal polynomial over $\K$. 
\end{Obs}

\begin{Lema}\label{lemMigualMadeltaParMinimaldeDefinicao}
	Suppose that $S\neq \emptyset$. If $a\in S$ has the smallest degree over $\K$ among elements in $S$, then $(a, \mu(x-a))$ is a minimal pair of definition for $\mu$. 
\end{Lema}

\begin{proof}
Take $b\in \overline{\K}$ such that $\mu(b-a)\geq \delta$. Thus, by Lemma~\ref{lemParesDefinemMesmaVal}, $\mu_{a, \delta}=\mu_{b,\delta}$. That is, $b\in S$. By the minimality of $a$, it follows that
	$$[\K(b):\K]\geq [\K(a):\K]  $$
	and then $(a,\mu(x-a))$ is a minimal pair for $\mu$. Therefore $(a,\delta)$ is a minimal pair of definition for $\mu$.
\end{proof}

In \cite{josneiKeyPolyMinimalPairs} it is proved the following relation between minimal pairs, optimizing roots and key polynomials.

\begin{Teo}[Theorem 1.1 of \cite{josneiKeyPolyMinimalPairs}]\label{teoparminimalpolichave} Let $Q\in \K[x]$ be a monic irreducible polynomial and choose   an optimizing root $a$ of $Q$. Then $Q$ is a key polynomial for $\nu$ if and only if $(a,\delta(Q))$ is a minimal pair for $\nu$. Moreover, $(a, \delta(Q))$ is a minimal pair of definition for $\nu$ if and only if $Q$ is a key polynomial and $\nu=\nu_Q$. 
\end{Teo}

We present two lemmas that show some properties of optimizing roots and minimal pairs. These lemmas will be very useful in Section~\ref{SecResultado1}.

\begin{Lema}\label{lemMuXmenosaQigualMuQQ}
	Let $f\in \overline{\K}[x]$ and let $(a,\delta(f))$ be a pair such that $a\in \overline{\K}$ is an optimizing  root of $f$. If $g\in \overline{\K}[x]$ is such that $\delta(g)<\delta(f)$, then
\[
\mu_{x-a}(g) = \mu(g) = \mu(g(a))\mbox{ and }\left(\frac{g}{g(a)}\right)\mu_{x-a}=1.
\]
Moreover, if $Q\in \K[x]$ is a key polynomial for $\nu$, then for an optimizing root $a$ of $Q$ we have
	$$\mu_{x-a}(Q) = \mu(Q).$$ 
\end{Lema}
\begin{proof}
Since $\overline \K$ is algebraically closed, it is enough to show the first part for $g(x)=x-b$ with $c\in \overline \K$. It follows from Lemma \ref{Lemaintermediario} that
\[
\mu_{x-a}(x-b)=\mu(x-b)\leq \mu(a-b)=\mu(g(a)).
\]
This and the fact that
\[
\mu(a-b)\geq\min\{\mu(x-a),\mu(x-b)\}
\]
gives us that $\mu(g(a))=\mu(a-b)=\mu(x-b)=\mu(g)$.

Since
\[
\mu(g-g(a))=\mu(x-a)>\mu(x-b)=\mu(a-b)=\mu(g(a))
\]
we have that
\[
\mu_{x-a}\left(\frac{g-g(a)}{g(a)}\right)>0\mbox{ and hence }0=\left(\frac{g-g(a)}{g(a)}\right)\mu_{x-a}.
\]
Consequently,
\[
\left(\frac{g}{g(a)}\right)\mu_{x-a}=\left(\frac{g(a)}{g(a)}\right)\mu_{x-a}=1.
\]

Now let $Q$ be a key polynomial and $a$ an optimizing root of $Q$. Let $a=c_1,\ldots,c_r$ be all the roots of $Q$. Then for every $i$, $1\leq i\leq r$, we apply Lemma \ref{Lemaintermediario} to obtain
\[
\mu(x-c_i)=\mu_{x-a}(x-c_i).
\]
Hence,
\[
\mu_{x-a}(Q)=\sum_{i=1}^r\mu_Q(x-c_i)=\sum_{i=1}^r\mu(x-c_i)=\mu(Q).
\]

\end{proof}

\begin{Lema}\label{lemGrauMenorDeltaMenor} 
	Let $(a,\gamma)$ be a minimal pair for $\nu$ with $\gamma=\mu(x-a)$. For all $f\in \K[x]$ with $\deg(f)< [\K(a):\K]$ we have $\delta(f)<\gamma$.
\end{Lema}

\begin{proof}
	For every root $b\in \overline{\K}$ of $f$, we know that $[\K(b):\K]\leq \deg(f)<[\K(a):\K]$. By the definition of minimal pair, we conclude that $\mu(b-a)<\gamma=\mu(x-a)$. Suppose that $b\in \overline{\K}$ is an optimizing root of $f$. Then,
	$$\delta(f)=\mu(x-b)=\mu(x-a+(a-b))=\mu(a-b)<\gamma. $$
\end{proof}

\begin{Obs}\label{obsGrauMenor}
	Let $g\in \K[x]$  be such that $\deg(g)<\deg(Q)$, with $Q$  a key polynomial. Taking an optimizing root $a\in \overline{\K}$ of $Q$, Theorem~ \ref{teoparminimalpolichave} says that $(a,\delta(Q))$ is a minimal pair and $\deg(Q)=[\K(a):\K]$. Then, by Lemma~\ref{lemGrauMenorDeltaMenor}, $\delta(g)<\delta(Q)$ and, by  Lemma~\ref{lemMuXmenosaQigualMuQQ},  $\mu_{x-a}(g)=\mu(g)=\mu(g(a))$.
\end{Obs}
We have now the main ingredients to define our main setting.
\begin{equation}                           \label{sit}
\left\{\begin{array}{ll}
\K & \mbox{is a field}\\
\overline \K & \mbox{is an algebraic closure of }\K\\
\nu & \mbox{is a Krull valuation on }\K[x]\\
\mu & \mbox{is a valuation on }\overline \K[x]\mbox{ extending }\nu\\
Q&\mbox{is a key polynomial for }\nu\\
a&\mbox{is an optimizing root for }Q.
\end{array}\right.
\end{equation}

\section{On an equality of valuations given by truncation }\label{SecResultado1}
The main goal of this section is to prove the following theorem.
\begin{Teo}\label{teorandretromeno}
Suppose that we are in the situation \eqref{sit}. Then
\[
\mu_{x-a}|_{\K[x]}=\nu_Q.
\]
\end{Teo}

We will divide the proof of Theorem \ref{teorandretromeno} in two cases: when $\nu(Q)\in \mu\overline{\K}$ and when $\nu(Q)\not\in \mu\overline{\K}$. We will prove some lemmas that will help us with the case $\nu(Q)\in \mu\overline{\K}$. In this case, we will present a transcendental element for $\K(x)\mu_Q\mid \K\nu$ and $\overline{\K}(x)\mu_{x-a}\mid \K\nu$

Our first result is a consequence of  Lemma 2.3 \textbf{(iii)}  from \cite{josneiKeyPolyPropriedades} applied to the valuation $\nu_Q$. 

\begin{Lema}\label{lemMQeProdhi}
In the situation \eqref{sit} we	have the following.
\begin{description}
\item[(i)] The polynomial $Q$ is a key polynomial for $\nu_Q$.
\item[(ii)] Take polynomials  $h_1,\ldots ,h_s\in \K[x]$  with $\deg(h_i)<\deg(Q)$ for every $i$, $1\leq i \leq s$. If
$$\prod_{i=1}^sh_i=lQ+ p$$
with $\deg(p)<\deg(Q)$, then
$$\nu_Q\left(\prod_{i=1}^sh_i\right)=\nu_Q(p)<\nu_Q(lQ). $$
\end{description}
\end{Lema} 

\begin{proof}
In order to prove \textbf{(i)} we note that if $\deg(f)<\deg(Q)$, then $\deg(\partial_bf)\leq \deg(f)<\deg(Q)$ for every $b$, $1\leq b \leq \deg(f)$. Hence, 
	$$\epsilon_Q(f) := \underset{1\leq b \leq \deg(f)}{\max}\left\lbrace \frac{\nu_Q(f)-\nu_Q(\partial_bf)}{b}\right\rbrace = \underset{1\leq b \leq \deg(f)}{\max}\left\lbrace \frac{\nu(f)-\nu(\partial_bf)}{b}\right\rbrace=\epsilon(f).$$

By the same reasoning and by the definition of $\nu_Q$, we obtain $\epsilon_Q(Q)=\epsilon(Q)$. Since $Q$ is a key polynomial for $\nu$, we have $\epsilon(f)<\epsilon(Q)$, which is the same as $\epsilon_Q(f)<\epsilon_Q(Q)$. Then $Q$ is a key polynomial for $\nu_Q$.

The second item follows by applying Lemma 2.3 \textbf{(iii)} of \cite{josneiKeyPolyPropriedades} for $\nu_Q$.
\end{proof}

\begin{Lema}\label{lemMuQCoefProdGrauMenorQ}
	Let $Q$ be a key polynomial for $\nu$ and take 
	$$f=f_0+f_1Q+\ldots +f_sQ^s $$
	where each $f_i$ is $0$ or a product of polynomials of degree smaller than $n=\deg(Q)$. Then,
	$$\nu_Q(f)=\underset{0\leq i \leq s}{\min}\{ \nu(f_iQ^i) \}. $$
\end{Lema}

\begin{proof}For each $i$, $0\leq i \leq s$, we write $f_i=q_iQ+r_i$ with $r_i=0$ or $\deg(r_i)<n$.  By Lemma~\ref{lemMQeProdhi} we have that
	$$\nu_Q(f_i)=\nu(f)=\nu_Q(r_i)<\nu_Q(q_iQ). $$
	Set $g=r_0+r_1Q+\ldots +r_sQ^s$ and take $m\in \N$ such that $\nu_Q(g)=\nu(r_mQ^m).$ Then, for every $i$, we  conclude that
	$$\nu_Q(q_iQ^{i+1})>\nu(r_iQ^i)\geq \nu(r_mQ^m). $$
	Since 
	$$f-g = \sum_{i=0}^sq_iQ^{i+1}, $$
	we obtain that 
	$$\nu_Q(f-g)>\nu(r_mQ^m)=\nu(g)\geq\nu_Q(g). $$	
	
	Hence, $\nu_Q(f)=\nu_Q(g)$. Consequently,
	$$\nu_Q(f)\geq \underset{0\leq i \leq s}{\min}\{ \nu(f_iQ^i)  \} =  \underset{0\leq i \leq s}{\min}\{ \nu(r_iQ^i)  \} = \nu_Q(g)=\nu_Q(f). $$
	Therefore, the result follows.	
\end{proof}

\begin{Prop}\label{lemrmuQtrans}
Let $Q$ be a key polynomial for $\nu$ and suppose there exists $e\in \N$ such that $\nu(Q^e)=\nu(h)$, where $h\in \K[x]$ and $\deg(h)<\deg(Q)$. Let $r=\frac{Q^e}{h}$.  Then the residue of $r$ in $\K(x)\nu_Q$  is transcendental  over $\K\nu$. In particular, the residue of $r$ in $\K(x)\nu_Q$  is transcendental  over any algebraic extension of $\K\nu$ contained in $\K(x)\nu_Q$. 
\end{Prop}

\begin{proof}
	We immediately see that  $\nu_Q(r)=0$. Suppose there are  $b_i\in \K$ such that $\nu(b_i)\geq 0$ for all $i$ and
	$$\sum_{i=0}^{s}(b_i\nu_Q)(r\nu_Q)^i=\sum_{i=0}^{s}\left(b_i\frac{Q^{ei}}{h^i}\right)\nu_Q=0. $$
	
	Then
	$$0=\sum_{i=0}^{s}\left(b_i\frac{Q^{ei}}{h^i}\right)\nu_Q=\left(\frac{h^sb_0+h^{s-1}b_1Q^e+\ldots +b_sQ^{es}}{h^s}\right)\nu_Q. $$
	
	Hence,
	$$\nu_Q(h^sb_0+h^{s-1}b_1Q^e+\ldots +b_sQ^{es})>\nu_Q(h^s). $$
	
	Suppose, aiming a contradiction, that there exists $j$ such that $b_j\nu_Q\neq 0$. Thus $\nu(b_j)=0$. Since
	$$\nu_Q(h^s)=se\nu(Q)  $$ 
	and
	\begin{align*}
	\nu_Q(h^{s-i}b_iQ^{ei}) &= (s-i)\nu_Q(h)+\nu_Q(b_i)+ei\nu_Q(Q)\\
	& = (s-i)e\nu(Q)+\nu_Q(b_i)+ei\nu(Q)\\
	& = \nu_Q(b_i)+se\nu(Q) \geq se\nu(Q),
	\end{align*}
we would have
	$$\nu_Q(h^s)=\underset{0\leq i \leq s}{\min}\{ \nu_Q(h^{s-i}b_iQ^{ei}) \} = \underset{0\leq i \leq s}{\min}\{ \nu(h^{s-i}b_iQ^{ei}) \}. $$
	
	Hence,
	$$\nu_Q(h^sb_0+h^{s-1}b_1Q^e+\ldots +b_sQ^{es})> \underset{0\leq i \leq s}{\min}\{ \nu(h^{s-i}b_iQ^{ei}) \}. $$
	
	This contradicts Lemma~\ref{lemMuQCoefProdGrauMenorQ}. Therefore, we must have $b_i\nu_Q=0$ for all $i$.  This means that $r\nu_Q\in \K(x)\nu_Q$ is transcendental over $\K\nu$.
\end{proof}

In order to prove that the residue of $r=Q^e/h$ in $\overline{\K}(x)\mu_{x-a}$  is also transcendental over $\overline\K\nu$ (Lemma~\ref{lemRMuXmenosAtrans}) we will use the following lemmas.

\begin{Lema}\label{lemMuQeDeltaQ}
	Let $a\in \overline{\K}$ be an optimizing root  of a polynomial $f\in \K[x]$. We have $\nu(f)\in \mu\overline{\K}$ if and only if $\delta(f)\in \mu\overline{\K}$.
\end{Lema}

\begin{proof}
	We write 
	$$f(x) = \prod_{i=1}^{k}(x-a_i)\prod_{j=k+1}^{n}(x-a_i) $$
with $a=a_1$ and 
	$$\delta(Q)=\mu(x-a_1)=\ldots = \mu(x-a_k)>\mu(x-a_i) $$
for all $j$ with $k+1\leq j \leq n$. Then, $\delta(f)>\delta(x-a_j)$ and, by Lemma~\ref{lemMuXmenosaQigualMuQQ}, $\mu(x-a_j)=\mu(a-a_j)\in\mu\overline{\K}$ for all $j$ with $k+1\leq j \leq n$. Thus, since
	$$\nu(f)=\mu(f)=k\delta(f)+\sum_{j=k+1}^n\mu(x-a_j) $$
we see that if $\delta(f)\in \mu\overline{\K}$, then $\nu(f)\in \mu\overline{\K}$. 
	
	On the other hand, if $\nu(f)\in \mu\overline{\K}$, then $k\delta(f)\in  \mu\overline{\K}$. Since $\mu\overline{\K}$ is a divisible group, we obtain that $\delta(f)\in \mu\overline{\K}$. 
\end{proof}

\begin{Lema}\label{lemresiduosMuXmenosB}
Suppose that we are in the situation \eqref{sit} and that $\nu(Q)\in \mu\overline{\K}$. Let $d_a\in \overline{\K}$ be such that $\mu(x-a)=\mu(d_a)\in \mu\overline{\K}$.  For $b\in \overline{\K}$, suppose that there exists $d_b\in\overline{\K}$ such that $\mu(x-b)=\mu(d_b)$. Then we have the following.	
\begin{description}
\item[(i)] The element $y=\dfrac{x-a}{d_a}\mu_{x-a}\in \overline{\K}(x)\mu_{x-a}$ is transcendental over $\overline{\K}\mu$.
\item[(ii)] If $\mu(x-b)\leq\mu(x-a)$, then
\[
\mu_{x-a}\left(\dfrac{x-b}{d_b}\right)\geq 0\mbox{ and }\dfrac{x-b}{d_b}\mu_{x-a}\in \overline{\K}\mu[y].
\]
\end{description}
\end{Lema}

\begin{proof}
The item \textbf{(i)} follows from Proposition~\ref{lemrmuQtrans} applied to $\mu$ and its key polynomial $x-a$.

In order to prove \textbf{(ii)} suppose $\mu(x-b)\leq\mu(x-a)$. By Lemma \ref{Lemaintermediario} we obtain that
\[
\mu(d_b)=\mu(x-b)=\mu_{x-a}(x-b)\leq\mu_{x-a}(x-a).
\]
Therefore,
\begin{displaymath}
\begin{array}{rcl}
\displaystyle\frac{x-b}{d_b}\mu_{x-a} &=& \displaystyle\frac{x-a+a-b}{d_b}\mu_{x-a}= \frac{x-a}{d_b}\mu_{x-a}+ \frac{a-b}{d_b}\mu_{x-a}\\[10pt]
&=&\displaystyle y\frac{d_a}{d_b}\mu_{x-a}+\frac{a-b}{d_b}\mu_{x-a}\in \overline{\K}\mu[y]
\end{array}
\end{displaymath}
\end{proof}

\begin{Obs}
In \textbf{(ii)} of the above Lemma, if $\mu(x-b)=\mu(x-a)$, then
\[
\deg_y\left(\displaystyle\frac{x-b}{d_b}\mu_{x-a}\right)=1.
\]
On the other hand, if $\mu(x-b)<\mu(x-a)$, then
\[
\deg_y\left(\displaystyle\frac{x-b}{d_b}\mu_{x-a}\right)=0.
\]
\end{Obs}
\begin{Prop}\label{lemRMuXmenosAtrans}
Suppose that we are in the situation \eqref{sit} and that there exists $e\in \N$ with $\nu(Q^e)=\nu(h)$, for $h\in \K[x]$ and $\deg(h)<\deg(Q)$. Let $r=\frac{Q^e}{h}$.  Then the residue of 
$r$  in $\overline{\K}(x)\mu_{x-a}$
is transcendental over $\overline{\K}\mu$. 
\end{Prop}

\begin{proof}
By Lemma~\ref{lemMuXmenosaQigualMuQQ} we can see that $\mu_{x-a}(r)=0$. Let $a=a_1,\ldots, a_r\in\overline \K$ be all the roots of $Q$. If $\mu(x-a_i)<\delta(Q)$, then by Lemma~\ref{lemMuXmenosaQigualMuQQ} we have
\[
\mu_{x-a}(x-a_i) = \mu(d_i)\mbox{ for }d_i:=a-a_i.
\]
On the other hand, if $\mu(x-a_i)=\delta(Q)$, then by Lemma~\ref{lemMuQeDeltaQ} (and our assuption that $\mu(Q)\in\mu\overline \K$) we have $\mu(x-a_i)=\delta(Q)\in \mu\overline{\K}$. Hence, there exists $d_i\in \overline{\K}$ such that $\mu(x-a_i)=\mu(d_i)$.

Set $y=\dfrac{x-a}{d_1}\mu_{x-a}$. We have that
\begin{equation}\label{equaparocaio1}
\left(\frac{Q^e}{h}\right)\mu_{x-a}=y^e\frac{d_1^e\cdots d_r^e}{h}\mu_{x-a}\prod_{i=2}^r\left(\frac{x-a_i}{d_i}\right)^e\mu_{x-a}.
\end{equation}
By Lemma~\ref{lemMuXmenosaQigualMuQQ} we have that
\begin{equation}\label{equaparocaio2}
\frac{d_1^e\cdots d_r^e}{h}\mu_{x-a}=\frac{d_1^e\cdots d_r^e}{h(a)}\mu\in\overline\K\mu.
\end{equation}
Also, by Lemma \ref{lemresiduosMuXmenosB} \textbf{(ii)}, for each $i$, $2\leq i\leq r$, we have
\begin{equation}\label{equaparocaio3}
\frac{x-a_i}{d_i}\mu_{x-a}\in \overline \K\mu[y].
\end{equation} 

By \eqref{equaparocaio1}, \eqref{equaparocaio2} and \eqref{equaparocaio3}, we obtain that $r\mu_{x-a}=yp(y)$ for some $p(y)\in\overline \K\mu[y]$. Since $y$ is transcendental over $\overline{\K}\mu$ (Lemma \ref{lemresiduosMuXmenosB} \textbf{(i)}) we conclude that $r\mu_{x-a}$ is also transcendental over $\overline{\K}\mu$. In particular, $r\mu_{x-a}$ is transcendental over $\K(a)\mu$ and over $\K\nu$.
\end{proof}

Since  $\K(a)$ is a simple extension of  $\K$ obtained by adjoining  $a$, all elements of this field extension have the form $g(a)$, where $g\in \K[x]$ and $\deg(g)<\deg(Q)=[\K(a):\K]$. Then, the map
\[
\K(a)\mu\lra\K[x]\nu_Q\mbox{ given by }f(a)\mu\longmapsto f(x)\nu_Q
\]
is a ring homormosphism because of Lemma \ref{lemMQeProdhi} \textbf{(ii)}. Moreover, it is injective by definition. Hence we will consider the embedding $\K(a)\mu\lra \K[x]\nu_Q$.

If we assume that $\nu(Q)\in \mu\overline{\K}$, since $\mu\overline \K$ is the divisible hull of $\nu \K$, there exists $e\in \N$ such that $e\nu(Q)\in \mu\K(a)$. We can take this positive integer as the least with this property. Thus, there exists $h\in \K[x]$ such that $\deg(h)<\deg(Q)$ and
$$\nu_Q(h)=\nu(h)=\mu(h(a)) = e\nu(Q). $$

\begin{Cor}\label{corRMuQeRMuXmenosAtrans}
	Let $Q$ be a key polynomial for $\nu$ and suppose $\nu(Q)\in \mu\overline{\K}$. Take $e\in \N$ the least positive integer such that $e\nu(Q)\in \mu\K(a)$. Choose $h\in \K[x]$ with $\deg(h)<\deg(Q)$ and such that $e\nu(Q)=\mu(h(a))$. Let $r=\frac{Q^e}{h}$. Then the elements 
	$$r\nu_{Q}\in \K(x)\nu_Q \text{ and } r\mu_{x-a}\in  \overline{\K}(x)\mu_{x-a}$$
	are transcendental over $\K(a)\mu$. 
\end{Cor}

\begin{proof}The proof  follows immediately from Proposition~\ref{lemrmuQtrans} and  Proposition~\ref{lemRMuXmenosAtrans}. 
	
\end{proof}

Consider now the following setting.
\begin{equation}                           \label{sit2}
\left\{
\begin{array}{l}
\mbox{We are in the situation \eqref{sit}, }\nu(Q)\in\mu\overline \K\mbox{, and}\\
\begin{array}{ll}
e\in\N & \mbox{is the least positive integer for which }e\nu(Q)\in \mu \K(a)\\
h\in \K[x] & \mbox{is such that }\nu(h)=e\nu(Q)\mbox{ and }\deg(h)<n=\deg(Q)
\end{array}
\end{array}\right..
\end{equation}

Using the element $r=Q^e/h$ we can prove, when $\nu(Q)\in \mu\overline{\K}$, that
\[
\mu_{x-a}(f)=\nu_Q(f)\mbox{ for every }f\in \K(r)\mbox{ or }f\in \K[x]\mbox{ with }\deg(f)<ne.
\]
These results are the next two lemmas.

\begin{Lema}\label{lemigualQuandoGrauMenoren}
Assume that we are in the situation \eqref{sit2}. If $g\in \K[x]$ is such that $\deg(g)<ne$ then $\mu_{x-a}(g)=\mu(g)=\nu_Q(g)$.
\end{Lema}

\begin{proof}
Let $g\in \K[x]$ with $\deg(g)<ne$. Then, its  $Q$-expansion is of the form
$$g = \sum_{i=0}^{e-1}g_iQ^i.$$
We claim that $\mu_{x-a}(g)= \underset{0\leq i \leq e-1}{\min}\{ \mu_{x-a}(g_i)+i\mu_{x-a}(Q)  \}$. Otherwise, there would exist $i_0$ and $i_1$, $0\leq i_0<i_1\leq e-1$, such that 
	$$\mu_{x-a}(g_{i_0})+i_0\mu_{x-a}(Q)=\mu_{x-a}(g_{i_1})+i_1\mu_{x-a}(Q). $$
Then, 
$$\mu_{x-a}(g_{i_0})-\mu_{x-a}(g_{i_1})=(i_1-i_0)\mu_{x-a}(Q)=(i_1-i_0)\nu(Q) $$
and since $\mu_{x-a}(g_i)=\mu(g_i(a))=\mu(g_i(a))\in \mu\K(a)$ for every $i$, $0\leq i\leq e-1$, this would imply that
\[
(i_1-i_0)\nu(Q) \in\mu\K(a).
\]
Since $i_1-i_0<e$, this would contradict the minimality of $e$. Therefore,
	$$\mu_{x-a}(g)= \underset{0\leq i \leq e-1}{\min}\{ \mu_{x-a}(g_i)+i\mu_{x-a}(Q)  \} =  \underset{0\leq i \leq e-1}{\min}\{ \nu(g_i)+i\nu(Q)  \}=\nu_Q(g). $$
\end{proof}

\begin{Lema}\label{lemQuocientefgalgebrico}
Assume that we are in the situation \eqref{sit2}. Take $f,g\in \K[x]$ such that $\deg(f),\deg(g)<ne$ and let $u=f/g$. If $\mu_{x-a}(u)=0$, then $u\mu_{x-a}$ is algebraic over $\K\nu$. 
\end{Lema}

\begin{proof}
Take $f\in \K[x]$ with $\deg(f)<ne$ and $Q$-expansion
\[
f=f_0+f_1Q+\ldots+f_{e-1}Q^{e-1}.
\]
Fix $i$, $0\leq i \leq e-1$, such that $\mu(f_iQ^i)=\min_{0\leq j\leq e-1}\{\mu(f_jQ^j)\}$. By the minimality of $e$ (and Lemma \ref{lemMuXmenosaQigualMuQQ}) we have that
\begin{equation}\label{equaimportagf}
\mu(f_jQ^j)>\mu(f_iQ^i)=\mu(f)\mbox{ for every }j\neq i.
\end{equation}
Take another polynomial $g\in \K[x]$of degree smaller than $ne$ with $Q$-expansion 
\[
g=g_0+g_1Q+\ldots+g_{e-1}Q^{e-1}.
\]
Also by the minimality of $e$ (and Lemma \ref{lemMuXmenosaQigualMuQQ}), if $\mu(f)=\mu(g)$, then
\begin{equation}\label{equaimportagg}
\mu(g_jQ^j)>\mu(g_iQ^i)=\mu(g)\mbox{ for every }j\neq i.
\end{equation}
Take $b,c\in \overline \K$ such that $\mu(b)=\mu(Q)$ and $\mu(c)=\mu(f_i)=\mu(g_i)$. Set
\[
\alpha:=\left(\frac{Q}{b}\right)\mu_{x-a},\beta:=\left(\frac{f_i}{c}\right)\mu_{x-a}\mbox{ and }\gamma:=\left(\frac{g_i}{c}\right)\mu_{x-a}.
\]
By Lemma \ref{lemMuXmenosaQigualMuQQ} we have
\[
\beta=\left(\frac{f_i(a)}{c}\right)\mu\in\overline{\K}\mu\mbox{ and }\gamma=\left(\frac{g_i(a)}{c}\right)\mu\in\overline{\K}\mu.
\]
Then, by \eqref{equaimportagf} and \eqref{equaimportagg}, we obtain
\[
\left(\frac{f}{cb^i}\right)\mu_{x-a}=\left(\frac{f_iQ^i}{cb^i}\right)\mu_{x-a}=\beta\alpha^i\mbox{ and }\left(\frac{g}{cb^i}\right)\mu_{x-a}=\left(\frac{g_iQ^i}{cb^i}\right)\mu_{x-a}=\gamma\alpha^i.
\]
Hence,
\[
r\mu_{x-a}=\frac{(f/cb^i)\mu_{x-a}}{(g/cb^i)\mu_{x-a}}=\frac{\beta\alpha^i}{\gamma\alpha^i}=\frac{\beta}{\gamma} \in\overline \K\mu.
\]
Since $\overline \K\mu$ is the algebraic closure of $\K\nu$ in $\overline \K(x)\mu_{x-a}$ the result follows.
\end{proof}

\begin{Lema}\label{lemigualEmKxrGrauMenorne}
Assume that we are in the situation \eqref{sit2}.  Let $g = t_0+t_1r+\ldots +t_sr^s,$
	where $t_i\in\K[x]$ and $\deg(t_i)<ne$ if $t_i\neq 0$. Then $$\mu_{x-a}(g)=\underset{0\leq i \leq s}{\min}\{ \mu_{x-a}(t_i)  \}.$$
	
\end{Lema}

\begin{proof}
 Let $m$, $0\leq m\leq s$, be such that
\[
\mu_{x-a}(t_m)=\min_{0\leq i\leq s}\{\mu_{x-a}(t_i)\}.
\]
In particular, $t_m\neq 0$. If $\mu_{x-a}(t_m)<\mu_{x-a}(g)$, then
\begin{displaymath}
\begin{array}{rcl}
0&=&\displaystyle\left(\frac{t_0+t_1r+\ldots+t_sr^s}{t_m}\right)\mu_{x-a}\\[10pt]
&=&\displaystyle\left(\frac{t_0}{t_m}\right)\mu_{x-a}+\ldots+(r\mu_{x-a})^m+\ldots+\left(\frac{t_s}{t_m}\right)\mu_{x-a}(r\mu_{x-a})^{s}.
\end{array}
\end{displaymath}
For every $i$, $0\leq i\leq s$, by assumption $\left(\dfrac{t_i}{t_m}\right)\mu_{x-a}$ is algebraic over $\K\nu$ and hence $\alpha$ would be algebraic over $\K\nu$.
\end{proof}

\begin{Cor}\label{lemigualEmKr}
Assume that we are in the situation \eqref{sit2}. Let $g = t_0+t_1r+\ldots +t_sr^s,$ where $t_i\in\K[x]$ and $\deg(t_i)<ne$ if $t_i\neq 0$. Then $$\mu_{x-a}(g)=\nu_{Q}(g).$$
In particular, $\mu_{x-a}=\nu_Q$ in $\K(r)$. 
\end{Cor}

\begin{proof}	By Lemma~\ref{lemigualEmKxrGrauMenorne},  $\mu_{x-a}(g)=\underset{0\leq i \leq s}{\min}\{ \mu_{x-a}(t_i) \}$. Applying Lemma~\ref{lemQuocientefgalgebrico} and Lemma~\ref{lemigualEmKxrGrauMenorne} to $\nu$ and $Q$ we conclude that $\nu_Q(g)=\underset{0\leq i \leq s}{\min}\{ \nu_Q(t_i) \}$. By Lemma~\ref{lemigualQuandoGrauMenoren}, since $\deg(t_i)<ne$, we have $\mu_{x-a}(t_i) = \nu_Q(t_i)$.  Thus the equality follows. 
\end{proof}

Now we have all the necessary tools to show the main result of this section. 

\begin{proof}[Proof of Theorem \ref{teorandretromeno}]
	
Assume first that $\nu(Q)\in \mu\overline{\K}$. Consider the field $\K(r)$, with $r$ as in Corollary~\ref {corRMuQeRMuXmenosAtrans}. The field extension $\K(x)\mid \K(r)$ is algebraic and have degree at most $ne$ (because $Q^e-rh=0$). Moreover, we can see $\K(x)$ as $\K(r)(x)$. Each element $f(x)\in \K(x)$ can be written as
	$$f(x)=\sum_{i=0}^{ne-1}f_i(r)x^i \mbox{ with }f_i(r)\in \K(r)\mbox{ for every }i, 0\leq i\leq ne-1.$$
For each $i$, $0\leq i\leq ne-1$, write
	$$f_i(r) = \frac{g_i(r)}{l(r)}, $$
with $g_i, l\in \K[r]$. Then,
	$$ f =\frac{g_0(r)+g_1(r)x+\ldots+g_{ne-1}x^{ne-1}}{l(r)}. $$
	Writing the numerator of $f$ as a polynomial in $r$ we obtain
	$$f = \frac{t_0(x)+t_1(x)r+\ldots+t_{s}(x)r^{s}}{l(r)}, $$
	where $t_i(x)\in \K[x]$, $\deg(t_i(x))<ne$ for every $i$, $0\leq i \leq s$. 
	We have
	$$\mu_{x-a}(f)=\mu_{x-a}(t_0+t_1r+\ldots+t_{s}r^{s})-\mu_{x-a}(l(r)).$$
We know, by Corollary~\ref{lemigualEmKr}, that
	$\mu_{x-a}(l(r))=\nu_Q(l(r))$ and
	$$\mu_{x-a}(t_0+t_1r+\ldots+t_{s}r^{s})=\nu_Q(t_0+t_1r+\ldots+t_{s}r^{s}).  $$
	Therefore  $\mu_{x-a}(f)=\nu_Q(f)$.

Suppose now that $\nu(Q)\not\in \mu\overline{\K}$. Since $\mu\overline\K$ is divisible we have that $\nu(Q)$ is torsion-free over $\mu\overline\K$. We will show that $\nu=\nu_Q$ and $\mu_{x-a}=\mu$ and consequently
\[
\mu_{x-a}|_{\K[x]}=\mu|_{\K[x]}=\nu=\nu_Q.
\] 

For any $f\in \K[x]$ with $\deg(f)<\deg(Q)$ we know, by Lemma~\ref{lemMuXmenosaQigualMuQQ}, that $\nu(f)=\mu(f(a))\in \mu\overline{\K}$. Therefore, applying Lemma~\ref{lemMuQLivreTorcaoMuigualMQ} for $\nu$ and $Q$, we have $\nu=\nu_Q$.

Since $\nu(Q)\notin \mu\overline{\K}$ we have, by Lemma~\ref{lemMuQeDeltaQ}, that $\delta(Q)\notin\mu\overline{\K}$. Since $\mu\overline{\K}$ is divisible we conclude that  $\delta(Q)=\mu(x-a)$ is torsion-free over $\mu\overline{\K}$. By Lemma~\ref{lemMuQLivreTorcaoMuigualMQ}, it follows that  $\mu_{x-a}=\mu$. This concludes our proof.
\end{proof}

\section{Valuation-transcendental and truncations on key polynomials}\label{SecResultado2}

In this section we prove Theorem~\ref{teoValTransiffTruncamento}. Suppose that $\mu$ is a valuation on $\overline{\K}[x]$ extending a valuation $\nu$ on $\K[x]$. We remember that $S$ is the set
$$S := \{ b\in\overline{\K}\mid \mu=  \mu_{b,\delta} \text{ where } \delta=\mu(x-b) \}.$$

\begin{Prop}\label{propSnaoVazioMuigualMuQ}
	If $S\neq \emptyset$, then there exists a key polynomial $Q\in \K[x]$ such that $\nu=\nu_Q$.
\end{Prop}
\begin{proof}
	Take $a\in S$ with the smallest degree over $\K$ among elements in $S$. By Lemma~\ref{lemMigualMadeltaParMinimaldeDefinicao}, $(a, \mu(x-a))$ is a minimal pair of definition for $\nu$. Let $Q\in \K[x]$ be the minimal polynomial of  $a$ over $\K$. By Remark~\ref{obsbSRaizOtimiza}, $a$ is an optimizing root of $Q$, that is, $\delta(Q)=\mu(x-a)$. By Theorem~\ref{teoparminimalpolichave}, $Q$ is a key polynomial for $\nu$ and  $\nu=\nu_Q$.
\end{proof}

Therefore, in order to prove Theorem \ref{teoValTransiffTruncamento}, Proposition~\ref{propSnaoVazioMuigualMuQ} tell us that it is enough to find a pair $(a,\delta)$ such that $\mu=\mu_{a,\delta}$.

We begin with the case when $\nu$ is a residue-transcendental valuation. This case was studied in \cite{popescu2} and \cite{popescu3}.  

If $\nu$ is  residue-transcendental then it is a Krull valuation and we can extend it naturally to $\K(x)$. We will need the following lemma.

\begin{Lema}\label{outroslema}
Let $\nu$ be any valuation on $\K[x]$. Assume that there exist $q\in \K[x]$ and $a\in \K$ such that $\xi=\left(\frac{q}{a}\right)\nu$ is transcendental over $\K\nu$. Then, for every $a_0,\ldots,a_s\in \K$ we have
\[
\nu(a_0+a_1q+\ldots+a_sq^s)=\min_{0\leq i\leq s}\{\nu(a_iq^i)\}.
\]
\end{Lema}
\begin{proof}
Suppose that there exist $a_0,\ldots,a_s\in \K$ such that
\[
\nu(a_0+a_1q+\ldots+a_sq^s)>\min_{0\leq i\leq s}\{\nu(a_iq^i)\}.
\]
Choose $l$, $0\leq l\leq s$, such that $\nu(a_lq^l)=\displaystyle\min_{0\leq i\leq s}\{\nu(a_iq^i)\}$. Then we obtain that
\[
\frac{a_0}{a_la^l}\nu+\ldots+\xi^l+\ldots+\frac{a_sa^{s-l}}{a_l}\nu\xi^s=\left(\frac{a_0+\ldots+a_lq^l+\ldots+a_sq^s}{a_la^l}\right)\nu=0.
\]
This is a contradiction to the fact that $\alpha$ is transcendental over $\K\nu$.
\end{proof}
The next result follows  from \cite{popescu2} (Proposition 1.1 and Proposition 1.3).
\begin{Prop}\label{propPopescu}
If the valuation $\nu$ in $\K[x]$ is residue-transcendental, then there exist $a\in \overline{\K}$ and  $\delta=\mu(x-a)\in  \mu\overline{\K}$ such that $\mu = \mu_{a,\delta}$.  	
\end{Prop}

\begin{proof}
Since $\nu$ is residue-transcendental there exists $r\in \K(x)$ such that $r\nu$ is transcendental over $\K\nu$. However $r\nu=r\mu\in \overline{\K}(x)\mu$. Thus, $r\mu$ is transcendental over $\K\nu$. Hence, by the transitivity of algebraicity, $r\mu$ is also transcendental over any algebraic extension of $\K\nu$ contained in $\overline{\K}(x)\mu$. Therefore, $r\mu$ is transcendental over $\overline{\K}\mu$, showing that $\mu$ is residue-transcendental.

Since $\mu$ is residue-transcendental, by the Abhyankar inequality and the fact that $\mu\overline\K$ is divisible, we have $\mu\overline{\K}(x)=\mu\overline{\K}$. Take $r=f/g\in \overline \K(x)$ such that $r\mu$ is transcendental over $\overline \K\mu$. Write
\[
f=c(x-a_1)\cdots (x-a_l)\mbox{ and }g=d(x-b_1)\cdots(x-b_m).
\]
For each $i$, $1\leq i\leq l$, and $j$, $1\leq j\leq m$, choose $e_i,f_j\in\overline \K$ such that
\[
\mu(x-a_i)=\mu(e_i)\mbox{ and }\mu(x-b_j)=\overline\mu(f_j).
\]
Then
\begin{equation}\label{equantrsanncenbaisxo}
r\mu=\beta\left(\prod_{i=1}^l\frac{x-a_i}{e_i}\mu\right)\left(\prod_{j=1}^m\frac{x-b_j}{f_j}\mu\right)^{-1}
\end{equation}
for some $\beta\in \overline \K\mu$. Since $r\mu$ is transcendental over $\overline \K\mu$,  by \eqref{equantrsanncenbaisxo}, we conclude that there exist $a,b\in\K$ such that $\dfrac{x-a}{b}\mu$ is transcendental over $\overline \K\mu$. By Lemma \ref{outroslema} we conclude that $\mu=\mu_{x-a}$.
\end{proof}

We now look at the case when $\nu$ is value-transcendental.
\begin{Lema}
Let $\nu$ be value-transcendental Krull valuation in $\K[x]$ and $\mu$ an extension of $\nu$ to $\overline{\K}[x]$. Then there exists $q\in \K[x]$ such that $\nu(q)\notin\mu\overline \K$.
\end{Lema}

\begin{proof}
Assume, aiming for a contradiction, that $\nu(q)\in\mu\overline\K$ for every $q\in \K[x]$. Since $\mu\overline \K$ is the divisible hull of $\nu\K$, $\nu(q)$ would be torsion over $\nu\K$, for every $q\in \K[x]$. Hence, $\nu$ is not value-transcendental, which is a contradiction.
\end{proof}

\begin{Prop}\label{lemValTransKbarraValMonomial}
Let $\mu$ be a Krull valuation in $\overline{\K}[x]$ which is value-transcendental over $\overline{\K}$. Then there exist $a\in \overline{\K}$ and $\delta=\mu(x-a)\in \mu\overline{\K}[x]$ such that  $\mu=\mu_{a,\delta}$. 
\end{Prop}

\begin{proof}
By the previous Lemma, there exists $q\in\overline \K[x]$ such that $\mu(q)\notin \mu\overline \K$. Since every polynomial in $\overline\K[x]$ can be written as the product of linear factors, we can assume that $q=x-a$ for some $a\in \overline \K$. Applying Lemma \ref{lemMuQLivreTorcaoMuigualMQ}, we obtain that $\mu=\mu_{a,\delta}$.
\end{proof}

We are now ready to prove Thereom \ref{teoValTransiffTruncamento}.

\begin{proof}[Proof of Theorem \ref{teoValTransiffTruncamento}]
Suppose that $\nu$ is valuation-transcendental.  	If $\nu$ is not Krull, then $\SU(\nu)=\langle Q\rangle $, with $Q$ a non-zero polynomial.  By definition $\epsilon(Q)=\infty$ and hence $Q$ is a key polynomial. Given  $f\in \K[x]$, if $f =f_0+\ldots+f_sQ^s$ is its $Q$-expansion, with $\deg(f_0)<\deg(Q)=\alpha$, then $\nu(f_0)\neq \infty$ and
		$$\nu(f) = \nu(f_0+\ldots+f_sQ^s)=\nu(f_0)=\nu_Q(f). $$
		
If $\nu$ is a Krull valuation, then take an extension $\mu$ of $\nu$ to $\overline\K[x]$. 	If $\nu$ is residue-transcendental, then by Proposition~\ref{propPopescu} $\mu$ admits a pair of definition. If $\nu$ is value-transcendental, then  by Proposition~\ref{lemValTransKbarraValMonomial} $\mu$ admits a pair of definition. In both cases, by Proposition~\ref{propSnaoVazioMuigualMuQ}, there exists a key polynomial $Q$ such that $\nu=\nu_Q$.

For the converse, assume that $\nu=\nu_Q$ for some key polynomial $Q$. By Theorem 1.3 of \cite{josneiKeyPolyMinimalPairs} we have that $\nu$ is valuation-transcendental. 
\end{proof}
\begin{Obs}
In \cite{Nart1} and \cite{Nart2} it is studied when a valuation admits \textit{MacLane-Vaqui\'e key polynomials}. Theorem \ref{teoValTransiffTruncamento} gives a criterion for such existence in terms of basic properties of valuations.
\end{Obs}
\section{On the graded algebra associated to a valuation}\label{gradegalh}
If $Q$ is a key polynomial for $\nu$, then $\nu_Q$ is a valuation in $\K[x]$ (Proposition 2.6 of \cite{josneiKeyPolyPropriedades}). On the other hand, if $\nu_q$ is a valuation, then it is not necessarily true that $q$ is a key polynomial (see Corollary 2.4 of \cite{josneiKeyPolyMinimalPairs}). In this section we assume that $\nu$ is a valuation on $\K[x]$ and $q\in K[x]$ is such that $\nu_q$ is a valuation. We study when $q$ is a key polynomial.

For each $\gamma\in \nu_q(\K[x])$, we consider the abelian groups
$$\mathcal{P}_\gamma = \{ f\in \K[x]\mid \nu_q(f)\geq \gamma \} \text{ and } \mathcal{P}_{\gamma}^{+} = \{ f\in \K[x]\mid \nu_q(f)>\gamma \}. $$

\begin{Def}
	The graded ring of $\K[x]$ associated to $\nu_q$ is defined as
	$$\mathcal G_q={\rm gr}_{\nu_q}(\K[x]):= \bigoplus_{\gamma\in \nu_q(\K[x])}\mathcal{P}_\gamma/ \mathcal{P}_{\gamma}^{+}.$$	
\end{Def}

For $f\not \in \SU(\nu_q)$ we will denote by $\inv_q(f)$ the image of $f$ in $\mathcal{P}_{\nu_q(f)}/ \mathcal{P}_{\nu_q(f)}^{+}.$ If $f\in \SU(\nu_q)$, then we set $\inv_q(f)=0$.

Let $R_q$ be the additive subgroup of $\mathcal G_q$ generated by
$$\{ \inv_q(f)\mid f\in \K[x]_d  \}, $$
where $d=\deg(q)$ and $\K[x]_d=\{ f\in \K[x]\mid \deg(f)<d \}$. Set $y:=\inv_q(q)$.

\begin{Lema}\label{Lematrasncey}
If 
$$a_0+a_iy+\ldots +a_sy^s=0 $$
for some $a_0, \ldots, a_r \in R_q$, then $a_i =0$ for every $i$, $0\leq i \leq s$. 
\end{Lema}
\begin{Obs}
If $R_q$ is a ring, then the previous lemma says that $y$ is transcendental over $R_q$.
\end{Obs}

\begin{proof}[Proof of Lemma \ref{Lematrasncey}]
Suppose there exist $a_0, \ldots, a_s \in R_q$ such that
	$$a_0+a_iy+\ldots +a_sy^s=0.$$
We can assume that $a_i=0$ or  $a_i=\inv_q(f_i)$ with $f_i\in \K[x]_d$ for every $i$, $0\leq i \leq s$. If $a_j\neq 0$ for some $j$, $0\leq j\leq s$, then $f_j\notin\SU(\nu_q)$. Set
	$$f=\sum_{i=1}^{s}f_iq^i. $$
By the assumption on the $a_i$'s and the definition of $\nu_q$ we have that
	$$0= \sum_{i\in S_q(f)}a_iy^i = \inv_q\left( \sum_{i\in S_q(f)}f_iq^i  \right) = \inv_q(f). $$
This is a contradiction because $f\notin \SU(\nu_q)$. 
\end{proof}

\begin{Def}
	For $f\in \K[x]$ with $q$-expansion $f=f_0+f_1q+ \ldots +f_rq^r$ we define
	$$S_q(f)=\{ i\mid \nu_q(f)=\nu(f_iq^i)  \} \text{ and } \delta_q(f)=\max S_q(f). $$
	
\end{Def}
\begin{Lema}\label{lemAlgGraduadaAnelPoliy}
	We have
	$$ \mathcal G_q= R_q[y].$$

\end{Lema}

\begin{proof}
	Take any $f\in \K[x]$ and write its $q$-expansion $f=f_0+f_1q+ \ldots +f_rq^r$, with $f_i\in \K[x]_d\cup \{0\}$ for every $i$, $0\leq i \leq r $. Then
	$$\nu_q\left(f-\sum_{i\in S_q(f)} f_iq^i \right) = \nu_q\left( \sum_{i\not\in S_q(f)} f_iq^i  \right) = \underset{i\not\in S_q(f)}{\min}\{ \nu(f_iq^i)  \}>\nu_q(f).$$
	
	Hence, 
	$$\inv_q(f) = \inv_q\left(\sum_{i\in S_q(f)} f_iq^i \right)=\sum_{i\in S_q(f)} \inv_q(f_i)y^i\in R_q[y]. $$
	Therefore, $\mathcal G_q= R_q[y].$ 
\end{proof}

\begin{Obs}
By Lemma \ref{Lematrasncey} and Lemma \ref{lemAlgGraduadaAnelPoliy} we see that the map
\[
\deg_y:\mathcal G_q\lra \N\cup\{0\}
\]
is well-defined. Moreover, if $f\in \K[x]$, then we see that $\deg_y(\inv_Q(f))=\delta_Q(f)$.
\end{Obs}

\begin{Teo}\label{tehoremkeyplo}
Suppose $\nu_q$ is a valuation on $\K[x]$. Then the following assertions are equivalent.
\begin{description}
\item[(i)] $q$ is a key polynomial for $\nu$.
	
\item[(ii)] For every $f,g\in \K[x]$ with $\deg(f),\deg(g)<\deg(q)$, if $fg=lq+r$ is the $q$-expansion of $fg$, then $\nu(fg)=\nu(r)<\nu(lq)$.
	
\item[(iii)] The set $R_q$ is a subring of $\mathcal G_q$.
	
\item[(iv)] For every $f,g\in \K[x]$ we have $\delta_q(fg)=\delta_q(f)+\delta_q(g)$.
	
\item[(v)] For every $f,g\in \K[x]$, if $\delta_q(f)=0=\delta_q(g)$ then $\delta_q(fg)=0$.

\end{description}
\end{Teo}

\begin{proof}

That \textbf{(i)} implies \textbf{(ii)} is proved in Lemma 2.3 \textbf{(iii)} of \cite{josneiKeyPolyPropriedades}. On the other hand, Corollary 3.52 of \cite{leloup} shows that \textbf{(ii)} implies \textbf{(i)}. It is well-known that \textbf{(ii)} implies \textbf{(iii)} (see Proposition 4.5 of \cite{josneimonomial}, for instance).

Assume that $R_q$ is a subring of $\mathcal G_q$. Since $\mathcal G_q$ is a domain so is $R_q$. Consequently, for $a,b\in \mathcal G_q$ we have $\deg_y(ab)=\deg_y(a)+\deg_y(b)$. Hence, for $f,g\in\K[x]$ we have
\[
\delta_q(fg)=\deg_y(\inv_q(fg))=\deg_y(\inv_q(f))+\deg_y(\inv_q(g))=\delta_q(f)+\delta_q(g).
\]

The assertion \textbf{(v)} follows immediately from \textbf{(iv)}. Assume that \textbf{(v)} is satisfied and take $f,g\in\K[x]$ with $\deg(f),\deg(g)<\deg(q)$. In particular, we have $\delta_q(f)=\delta_q(g)=0$. By our assumption, we have $\delta_q(fg)=0$. This implies that, if
\[
fg=lq+r\mbox{ is the }q\mbox{-expansion of }fg,
\]
then $\nu(fg)=\nu_q(fg)=\nu(r)<\nu(lq)$ (because $S_q(fg)=\{0\}$). This concludes the proof.
\end{proof}

In Theorem 5.3 of \cite{Andrei}, an alternative proof of Theorem \ref{teorandretromeno} is presented. This proof uses the structure of the graded algebra associated to the truncation in key polynomials. We present a brief comparison between these two proofs.

We assume that we are in the situation \eqref{sit}. Then, by Theorem \ref{tehoremkeyplo} and Lemmas \ref{Lematrasncey} and \ref{lemAlgGraduadaAnelPoliy} we have that
\begin{equation}\label{equatalgngrad}
\mathcal G_q=R_Q[y]\mbox{ and }\mathcal G_{x-a}=\left({\rm gr}_\mu\overline \K\right)[z]
\end{equation}
where $y=\inv_Q(Q)$ and $z=\inv_{x-a}(x-a)$. Since $\nu_Q(f)\leq \mu_{x-a}(f)$ for every $f\in \K[x]$ there exists a natural homomorphism
\[
\Phi:\mathcal G_q\lra \mathcal G_{x-a}
\]
given by
\begin{displaymath}
\Phi(\inv_Q(f))=\left\{\begin{array}{cl}
\inv_{x-a}(f)&\mbox{ if }\nu_Q(f)=\mu_{x-a}(f)\\
0&\mbox{ if }\nu_Q(f)<\mu_{x-a}(f)
\end{array}\right..
\end{displaymath}
Hence, in order to show Theorem \ref{teorandretromeno}, it is enough to show that $\Phi$ is injective.

They show that $\Phi(R_Q)\subseteq {\rm gr}_\mu\overline \K$ and that $\Phi|_{R_Q}$ is injective. This result is equivalent to our Lemma \ref{lemMuXmenosaQigualMuQQ}. Then, it is shown that $z\mid \inv_{x-a}(Q)$ (this result is equivalent to our Proposition \ref{lemRMuXmenosAtrans}). Finally, they use the characterization \eqref{equatalgngrad} and the fact that $\deg_z(ab)=\deg_z(a)+\deg_z(b)$ to conclude that $\Phi$ is injective.

\pagebreak

\noindent{\footnotesize JOSNEI NOVACOSKI\\
Departamento de Matem\'atica--UFSCar\\
Rodovia Washington Lu\'is, 235\\
13565-905, S\~ao Carlos - SP, Brazil.\\
Email: {\tt josnei@ufscar.br} \\\\

\noindent{\footnotesize CAIO HENRIQUE SILVA DE SOUZA\\
Departamento de Matem\'atica--UFSCar\\
Rodovia Washington Lu\'is, 235\\
13565-905, S\~ao Carlos - SP, Brazil.\\
Email: {\tt caiohss@estudante.ufscar.br} \\\\

\end{document}